\documentclass[11pt,reqno]{amsart}

\usepackage[mathscr]{eucal}
\usepackage{amsmath,amssymb,amsfonts,amsthm,enumerate}
\usepackage{hyperref}
\usepackage{color}
\usepackage{ragged2e}
\justifying\let\raggedright\justifying

\newtheorem{theorem}{Theorem}[section]
\newtheorem{lemma}[theorem]{Lemma}
\newtheorem{corollary}[theorem]{Corollary}
\newtheorem{proposition}[theorem]{Proposition}

\theoremstyle{remark}
\newtheorem{remark}[theorem]{Remark}
\newtheorem{example}[theorem]{Example}

\newcommand{\N}{\mathbb N}
\newcommand{\Z}{\mathbb Z}
\newcommand{\R}{\mathbb R}
\newcommand{\Q}{\mathbb Q}
\newcommand{\F}{\mathbb F}

\DeclareMathOperator{\ord}{ord}

\newcommand{\be}{\begin{equation}}
\newcommand{\ee}{\end{equation}}

\newcommand{\ber}{\begin{eqnarray}}
\newcommand{\eer}{\end{eqnarray}}

\DeclareSymbolFont{goo}{OMS}{cmsy}{b}{n}
\DeclareMathSymbol{\gooT}{\mathalpha}{goo}{"1}

\begin{document}

\title[divisibility on point counting over finite Witt rings]{
divisibility on point counting over finite Witt rings}

\author[W. Cao]{Wei Cao}
\address{School of Mathematics and Statistics\\ Minnan Normal University\\ Zhangzhou\\ Fujian 363000\\ P.R. China }
\email{caow2286@mnnu.edu.cn}

\author[D. Wan]{Daqing Wan}
\address{Department of Mathematics \\ University of California\\Irvine, 92697-3875\\ USA     }
\email{dwan@math.uci.edu}


\begin{abstract}
Let $\mathbb{F}_q$ denote the finite field of $q$ elements with characteristic $p$. Let $\mathbb{Z}_q$ denote
the unramified extension of the $p$-adic integers $\mathbb{Z}_p$ with residue field $\mathbb{F}_q$.
In this paper, we investigate the $q$-divisibility for the number of solutions of a polynomial system in $n$ variables over the finite Witt ring $\mathbb{Z}_q/p^m\mathbb{Z}_q$, where the $n$ variables of the polynomials are restricted to run through a combinatorial box lifting $\mathbb{F}_q^n$.
The introduction of the combinatorial box makes the problem much more complicated. We prove a $q$-divisibility theorem
for any box of low algebraic complexity, including the simplest Teichm\"uller box.  This extends the classical Ax-Katz theorem over finite field $\mathbb{F}_q$ (the case $m=1$).
Taking $q=p$ to be a prime, our result extends and improves a recent combinatorial theorem of Grynkiewicz. Our different approach is based on the addition operation of Witt vectors and is conceptually much more transparent.

\end{abstract}

\maketitle

\section{Introduction}\label{sect1}

Let $\N$ denote the set of positive integers. Let $p$ be a prime number and $q=p^h$ with $h\in \N$. Let $\mathbb{F}_q$ denote the finite field of $q$ elements and $\mathbb{F}_q[x_1,\ldots,x_n]$ the ring of the polynomials in $n$ variables $x_1,\ldots,x_n$ with coefficients in $\mathbb{F}_q$. The study of the common zeros of a system of polynomials in $\mathbb{F}_q[x_1,\ldots,x_n]$ is a classical and important subject in Number Theory and Arithmetic Geometry. In general it is hard to know the exact cardinality of the set of such common zeros in $\mathbb{F}_q$. However, the Chevalley-Warning and Ax-Katz theorems provide the estimates of $p$-divisibility for this problem by utilizing the degrees of the associated polynomials. Given a set $S$, let $|S|$ denote the cardinality of $S$. Write $X:=(x_1,\ldots,x_n)$ and set $[a,b]:=\{ x\in \Z \mid a\leq x\leq b\}$ for $a,b\in \R$.

\begin{theorem}{\rm (Chevalley-Warning)}\label{cwthm}
Let $f_1(X),\dots,f_s(X)\in \mathbb{F}_q[x_1,\ldots,x_n]$ be a system of nonzero polynomials, and let
\begin{equation*}
  V:=\{X\in \mathbb{F}_q^n | f_k(X)=0 \mbox{ for all } k\in [1,s]\}.
\end{equation*}
 If $n>\sum_{k=1}^s\deg (f_k)$, then $p$ divides $|V|$.
\end{theorem}
The Chevalley-Warning theorem also gave an affirmative answer to Artin's conjecture for the homogeneous polynomials (see \cite{Chevalley} and \cite{Warning}), and it was greatly improved by Ax \cite{Ax} for the case $s=1$ and Katz \cite{Katz} for general $s\geq 1$. Let $\mathrm{ord}_q $ denote the $q$-adic additive
valuation normalized by $\mathrm{ord}_qq=1$. If $q=p$, then $\mathrm{ord}_p $ is the $p$-adic additive
valuation normalized by $\mathrm{ord}_pp=1$. For $t\in\mathbb{R}$, let $\lceil t \rceil^*$ denote the least nonnegative
integer more than or equal to $t$. The Ax-Katz theorem can be stated as follows.

\begin{theorem}{\rm (Ax-Katz)}\label{axkatzthm}
With the same assumption as in Theorem \ref{cwthm}, we have
\begin{equation}\label{11}
  \mathrm{ord}_q(|V|)\geq \left\lceil\frac{n-\sum_{k=1}^s\deg (f_k)}{\max\nolimits_{k \in [1,s]}\deg (f_k)}\right\rceil^*.
\end{equation}
\end{theorem}

An elementary proof of the Ax-Katz theorem is given in \cite{Wan1}. The simplest proof of the Ax-Katz theorem and its extension to character sums are given in \cite{Wan2}. A reduction of the Ax-Katz theorem for a system of equations to Ax's theorem for a single equation has been found by Hou \cite{Hou}. Besides these, there has been a lot of research work on this topic, including extensions, refinements, variants and alternative proofs (see, for example, \cite{AS,AM,clark2,brink,cao1,cao2,cao3,cao4,castro,chencao,Clark1,clark3,hb,Moreno,moreno2,Wilson}).

Recently, motivated by combinatorial applications, Grynkiewicz \cite{Grynkiewicz} proved a version of the Chevalley-Warning and Ax-Katz theorems over the residue class ring $\Z_p/p^m\Z_p$, in which the varying prime power moduli are allowed.
\begin{theorem}\textup{(\cite[Theorem 1.3]{Grynkiewicz})}\label{grythm}
Let $p$ be a prime number and $\mathcal B=\mathcal I_1\times \cdots\times \mathcal I_n$ with each $\mathcal I_j\subseteq \Z_p$  a complete system of residues modulo $p$ for $j\in[1,n]$. Let $m_1,\ldots,m_s\in \N$ and $f_1,\ldots, f_s\in \mathbb \Z_p[x_1,\ldots,x_n]$ be a system of nonzero polynomials, and let
\begin{align*}V:=\{X\in \mathcal B:\; f_k(X)\equiv 0\pmod {p^{m_k}}\mbox{ for all } k\in [1,s]\}.\end{align*}
Then
\begin{equation*}
\mathrm{ord}_p(|V|)\geq \left\lceil\frac{n-\sum_{k=1}^s\frac{p^{m_k}-1}{p-1}\deg (f_k)}{\max\nolimits_{k \in [1,s]}\{p^{m_k-1}\deg (f_k)\}}\right\rceil^*.
\end{equation*}
\end{theorem}

Note that each $\mathcal I_j$ is a lifting of the prime field $\mathbb{F}_p$ in $\Z_p$ and thus the box $\mathcal B$ is a lifting of $\mathbb{F}_p^n$ in $\Z_p^n$.
Obviously, if $m_1=\cdots=m_s=1$, then Theorem \ref{grythm} recovers the Ax-Katz theorem for the prime finite field $\F_p$. The box $\mathcal B$ in Theorem \ref{grythm} allows many combinatorial applications. As suggested by Grynkiewicz in \cite{Grynkiewicz}, if there is an $m_k>1$ for some $k\in[1,s]$, one should appropriately choose the box $\mathcal B$ to apply Theorem \ref{grythm} to some problems in Combinatorial Number Theory. In other words, the elements in $\mathcal I_i$ should satisfy the proposition below. In our terminology, this just means that one should typically choose the
Teichm\"uller box.

\begin{proposition}\textup{(\cite[Proposition 1.4]{Grynkiewicz})}\label{prop}
Let $p$ be a prime  number and $m\in \N$. There exists a complete system of residues $\mathcal I\subseteq [0,p^m-1]$  modulo $p$ such that $$x^{p-1}\equiv
\left\{
  \begin{array}{ll}
    1 \pmod {p^m} & \hbox{if $x\not\equiv 0\pmod p$} \\
    0 \pmod {p^m} & \hbox{if $x\equiv 0\pmod p$,}
  \end{array}
\right.\quad\mbox{ for every $x\in \mathcal I$}.$$
\end{proposition}

To prove and apply Theorem \ref{grythm}, Grynkiewicz \cite{Grynkiewicz} comprehensively utilized the weighted Weisman-Fleck congruence \cite{SunWan}, Wilson's arguments \cite{Wilson}, etc. His method is combinatorial in nature, and it is not clear how to use his method to extend Theorem \ref{grythm} from $\Z_p$ to $\Z_q$ with the box $\mathcal B$ being a lifting of $\mathbb{F}_q^n$ so that it would also include the general Ax-Katz theorem. In fact, we will give counter-examples showing
that the $\Z_q$ generalization of Theorem \ref{grythm} is false.  This suggests that the problem is more subtle for
$\Z_q$ than for $\Z_p$.

Another restriction in Theorem \ref{grythm} is that the box is in split form, that is, the $n$-dimensional box $\mathcal B$ is
the product of one dimensional boxes $\mathcal I_j$ for $1\leq j\leq n$.
In general, a box (a lifting of $\mathbb{F}_p^n$ in $\Z_p^n$) will not be in
such a split form.  We will also give counter-examples showing that Theorem \ref{grythm} is false for general non-split boxes.

Despite all these obstacles, our aim of this paper is to investigate the problem over $\Z_q$
and a general box $\mathcal B$ lifting $\mathbb{F}_q^n$, in an attempt to unify and hence extend
both the general Ax-Katz theorem and Theorem \ref{grythm}. This desired unification is achieved in this paper.
Our main result says that the desired $q$-divisibility theorem holds over $\Z_q$
as long as the box $\mathcal B$ (lifting $\mathbb{F}_q^n$) has low algebraic complexity, in the sense that it is
close to the Teichm\"uller box up to a low degree polynomial perturbation.
In the case $q=p$, any split box has low algebraic complexity, which explains Theorem \ref{grythm}.
There are many non-split boxes with low algebraic complexity, and thus our result significantly extends Theorem \ref{grythm}
as well, even in the case $q=p$.
We now make these more precise.
For simplicity of exposition, we only state some weaker but simpler consequences of our main result in this introduction.

Let $T_q$ be the set of Teichm\"{u}ller representatives of $\mathbb{F}_{q}$ in $\mathbb{Z}_q$.
The set $T_q^n$ is clearly a lifting of $\mathbb{F}_q^n$, and is called the Teichm\"{u}ller box.
It is the simplest and nicest box for our purpose. Our result for the Teichm\"uller box is the following
statement.


\begin{theorem}[Corollary \ref{cor3}]\label{cor1intro}
Let $p$ be a prime number and $q=p^h$ with $h\in \N$. Let $f_1,\dots, f_s\in \mathbb{Z}_q[x_1,\ldots,x_n]$ be a system of nonzero polynomials. For given $m_1,\dots,m_s\in \N$, let
\begin{equation*}
  V:=\{X\in T_q^n \mid  f_k(X)\equiv 0 \pmod {p^{m_k}} \mbox{ for all } k\in [1,s]\}.
\end{equation*}
Then
\begin{equation*}
\mathrm{ord}_q(|V|)\geq \left\lceil\frac{n-\sum_{k=1}^s\frac{p^{m_k}-1}{p-1}\deg (f_k)}{\max\nolimits_{k \in [1,s]}\{p^{m_k-1}\deg (f_k)\}}\right\rceil^*.
\end{equation*}
\end{theorem}

In the case $m_1=\cdots = m_s=1$, this reduces to the Ax-Katz theorem over $\mathbb{F}_q$.
In fact, our proof in the general case reduces to this special case.

The possible extension from the Teichm\"uller box to a general box is more subtle.
Let us define a box $\mathcal B$ to be a subset of $\Z_q^n$ with $q^n$ elements such that $\mathcal B$ modulo $p$ is equal to $\F_q^n$. That is, $\mathcal B$ is a complete system of representatives of
$\F_q^n$ in $\Z_q^n$, equivalently, $\mathcal B$ is a lifting of $\F_q^n$ in $\Z_q^n$. In order to apply algebraic methods,
we would like to give an algebraic presentation of the combinatorial box $\mathcal B$, using the image of a polynomial map, following the spirit in \cite{Lai}.
As proved in Section \ref{sect4}, for any box $\mathcal B$, there exists a unique system of polynomials $g_{j}(X)\in\Z_q[x_1,\dots, x_n]$  ($1\leq j\leq n$) whose degree in each variable is at most $q-1$ such that for any $Y=(y_1, \dots, y_n)\in \mathcal B$, we have
\begin{equation}\label{zz1intro}
     Y=X + (g_{1}(X),\dots,g_{n}(X))p,
\end{equation}
where $X=(x_1, \dots, x_n)\in T_q^n$ is the Teichm\"uller lifting of the modulo $p$ reduction of $Y$.
In other words, the box $\mathcal B$ is simply the image of the Teichm\"uller box $T_q^n$ under the polynomial map
$X \longrightarrow X + (g_{1}(X),\dots,g_{n}(X))p$. This polynomial representation of the box $\mathcal B$ is unique since we require the
polynomials $g_j(X)$ to be reduced and thus have degrees at most $q-1$ in each variable, that is, we have reduced the polynomials modulo
the ideal $(x_1^q-x_1, \cdots, x_n^q-x_n)$. The total degree of $g_j$ is then bounded by $(q-1)n$.
The box $\mathcal B$ is called in split form or a split box if $\mathcal B = \mathcal I_1 \times \cdots \times \mathcal I_n$, where each $\mathcal I_j$
is a $1$-dimensional box in $\Z_q$ lifting $\mathbb{F}_q$. The box $\mathcal B$ is in split form, if and only if \eqref{zz1intro} becomes
\begin{equation*}
     Y=X + (g_{1}(x_1),\dots, g_{n}(x_n))p,
\end{equation*}
where each $g_{j}(x_j)$ depends only on the one variable $x_j$. In this case, each $g_j$ has total degree at most $q-1$,
much smaller than $n(q-1)$.

The degrees of the representing polynomials $g_j$'s provide a crude measure for the algebraic complexity of the
combinatorial box $\mathcal B$, see \cite{Lai} for a discussion of this in the case of finite fields. A random box $\mathcal B$ has high algebraic complexity, and hence algebraic
methods have limited values. One expects that a box of low algebraic complexity has some algebraic structure
and hence suitable for study using algebraic methods. This explains why we need a low degree bound on
the representing polynomials $g_j$ for the box $\mathcal B$ in the following theorems.

A polynomial $f\in\Z_q[x_1,\dots, x_n]$ is called a Teichm\"uller polynomial if all of its coefficients are
Teichm\"uller elements in $T_q$. Clearly, any polynomial $f\in\Z_q[x_1,\dots, x_n]$ has the unique expansion
$$f(X) = \sum_{i=0}^{\infty} p^i f_i(X),$$
where each $f_i(X)$ is a Teichm\"uller polynomial. This is called the Teichm\"uller expansion of $f$.
It is obtained from the Teichm\"uller expansion of the coefficients of $f$. For $x\in \R$, let $\lfloor x \rfloor$ denote the greatest integer less then or equal to $x$. Our result for a general box $\mathcal B$ in $\Z_q^n$ is as follows.

\begin{theorem}[Corollary \ref{thm44}]\label{introthm2}
Let $p$ be a prime number and $q=p^h$ with $h\in \N$. Let $\mathcal B$ be a general box in $\Z_q^n$ defined by the reduced polynomials $g_j \in \Z_q[x_1,\cdots, x_n]$ with $j\in [1,n]$ as above. Let $f_1,\dots, f_s\in \mathbb{Z}_q[x_1,\ldots,x_n]$ be a system of nonzero polynomials. For given $m_1,\dots,m_s\in \N$, let
\begin{equation*}
  V:=\{X\in \mathcal B \mid  f_k(X)\equiv 0 \pmod {p^{m_k}} \mbox{ for all } k\in [1,s]\}.
\end{equation*}
For $1\leq j\leq n$, write the Teichm\"uller expansion $pg_j(X) = \sum_{i=1}^{\infty} p^i g_{ij}(X)$. If $\deg(g_{ij})\leq
p^{h\lfloor\frac{i}{h}\rfloor}$ for all $j\in [1, n], i \in [1, m-1]$,
 then
\begin{equation*}
\mathrm{ord}_q(|V|)\geq \left\lceil\frac{n-\sum_{k=1}^s\frac{p^{m_k}-1}{p-1}\deg (f_k)}{\max\nolimits_{k \in [1,s]}\{p^{m_k-1}\deg (f_k)\}}\right\rceil^*.
\end{equation*}
\end{theorem}

If $\mathcal B$ is the Teichm\"uller box, then $g_{ij}=0$ for all $i, j\geq 1$, and the condition $\deg(g_{ij})\leq
p^{h\lfloor\frac{i}{h}\rfloor}$ is trivially satisfied. Theorem \ref{introthm2} is thus a generalization of
Theorem \ref{cor1intro} from the Teichm\"uller box to a general box of low algebraic complexity.

In the case $q=p$ and thus $h=1$, we obtain the following simpler consequence.

\begin{theorem}
[Corollary \ref{cor8}]\label{cor2intro}
Let $p$ be a prime number. Let $\mathcal B$ be a general box as defined above by the polynomials $g_j(x_1,\cdots, x_n)
\in \Z_p[x_1,\cdots, x_n]$ with $j\in [1,n]$.
Let $f_1,\dots, f_s\in \mathbb{Z}_p[x_1,\ldots,x_n]$ be a system of nonzero polynomials. For given $m_1,\dots,m_s\in \N$, let
\begin{equation*}
  V:=\{X\in \mathcal B \mid  f_k(X)\equiv 0 \pmod {p^{m_k}} \mbox{ for all } k\in [1,s]\}.
\end{equation*}
If $\deg(g_{j})\leq p$ for all $j\in [1,n]$, then
\begin{equation}\label{oc1}
\mathrm{ord}_p(|V|)\geq \left\lceil\frac{n-\sum_{k=1}^s\frac{p^{m_k}-1}{p-1}\deg (f_k)}{\max\nolimits_{k \in [1,s]}\{p^{m_k-1}\deg (f_k)\}}\right\rceil^*.
\end{equation}
\end{theorem}
In this theorem, if the box $\mathcal B$ is in split form, then $\deg(g_{j})\leq p-1$ and hence
the degree condition $\deg(g_{j})\leq p$
is automatically satisfied. In particular, \eqref{oc1} holds true for all split boxes $\mathcal B$, recovering Theorem \ref{grythm}.
Note that the above theorem is also true for many non-split boxes as long as the degrees of the $g_j$ are bounded by $p$.

We emphasize that Theorems \ref{cor1intro}, \ref{introthm2} and \ref{cor2intro} presented above are simpler weaker versions of our results. For their strong versions, which depend on the degree bounds in the $p$-adic expansion of the polynomials $g_j$'s and $f_k$'s, see Theorems \ref{thm2} and \ref{thm4}. Our basic idea is to use the addition operation of Witt vectors to reduce the congruence solution counting
in the box $\mathcal{B}$ to point counting of a system of equations over $\mathbb{F}_q$ for which the Ax-Katz theorem
can be applied. The key is to control the degrees of the resulting polynomial equations over $\mathbb{F}_q$. This
leads to the assumption on the degree bounds for the $g_j$'s, or more generally the degree bounds in the $p$-adic
Teichm\"uller expansions of the polynomials $g_j$'s and $f_k$'s.

The paper is organized as follows. Some basic knowledge about the Witt vectors is reviewed in Section \ref{sect2}. Then we apply the ring of Witt vectors over $\F_q$ to study the polynomials in $\mathbb{Z}_q[x_1,\ldots,x_n]$, which is divided into two parts: the generalization of Theorem \ref{grythm} to $\Z_q[x_1,\ldots,x_n]$ for the Teichm\"{u}ller box case is given in Section \ref{sect3}, and that for the general box is given in Section \ref{sect4}. At the end of this paper, we give examples showing that all the theorems
are false without the degree bounds of the representing polynomials $g_j$'s.

\section{Prerequisites}\label{sect2}

Witt vector rings and their variants are a useful tool in many branches of mathematics ranging from algebra and algebraic number theory to arithmetic geometry and homotopy theory. In this section, we only review the construction and simple properties of the classical $p$-typical Witt vectors of Witt and Teichm\"uller \cite{Witt}; for generalized or big Witt vectors, refer to \cite{Cartier,rabinoff}. Using the $p$-typical Witt vectors one may pass from a perfect field $K$ of characteristic $p$ to unramified complete discrete valuation ring with the residue field $K$ and quotient field of characteristic zero.

\subsection{The ring of $p$-typical Witt vectors}
Let $R$ be a commutative ring with identity and $\mathbb N_0=\{0\}\cup \N$. The underlying set of the ring of $p$-typical Witt vectors over $R$ is the set
$$W(R)=R^{\mathbb N_0}=\{(a_0,a_1,\dots) \mid a_i\in R\}.$$
Now we explore the mysterious algebraic structure of $W(R)$. For $n\in \mathbb N_0$, the $n$-th Witt polynomial is defined to be
\begin{equation}\label{wittpoly}
\omega_n(x_0,x_1,\dots,x_n):=x_0^{p^n}+px_1^{p^{n-1}}+\cdots+p^nx_n.
\end{equation}
\begin{remark}\label{homfact}
  If we define $\mathrm{wt}(x_i)=p^i$, then $\omega_n$ is weighted homogeneous of weighted degree $p^n$.
\end{remark}
Using the Witt polynomials, we can establish the so-called ghost (or phantom) map
  \begin{equation}\label{ghostmap}
     \omega: W(R) \rightarrow R^{\mathbb N_0}, \quad \mathbf a=(a_0,a_1,\dots) \mapsto \omega(\mathbf a)=(\omega_0(a_0),\omega_1(a_0,a_1),\dots),
  \end{equation}
  where $\omega_n(a_0,a_1,\dots,a_n)$ is called the $n$-th ghost (or phantom) component of $\mathbf a$. The ring $W(R)$ of $p$-typical Witt vectors over $R$ is defined by componentwise addition and multiplication via the ghost components, which was found by the pioneering and ingenious work of Witt \cite{Witt}. Let $\oplus$ and $\odot$ denote the addition and multiplication in the ring $W(R)$, respectively.
\begin{theorem}[Witt]\label{wittthm}
  There are two families of polynomials with integer coefficients
  \begin{align*}
    S_n(x_0,y_0;x_1, y_1;\dots;x_n,y_n),\quad  M_n(x_0,y_0;x_1, y_1;\dots;x_n,y_n), \quad n\in \mathbb N_0,
  \end{align*}
  such that for $\mathbf{a}=(a_0,a_1,\dots), \mathbf{b}=(b_0,b_1,\dots) \in W(R)$, we have
  \begin{itemize}
    \item[(i)]$\mathbf{a}\oplus\mathbf{b} = (S_0(a_0,b_0), S_1(a_0,b_0;a_1,b_1), \dots)$,
    \item[(ii)]$\mathbf{a}\odot\mathbf{b} = (M_0 (a_0,b_0), M_1(a_0,b_0;a_1,b_1), \dots)$,
    \item[(iii)] $\omega(\mathbf{a}\oplus\mathbf{b})=\omega(\mathbf{a})+\omega(\mathbf{b})$,
    \item[(iv)]$\omega(\mathbf{a}\odot\mathbf{b})=\omega(\mathbf{a})\omega(\mathbf{b})$.
  \end{itemize}
\end{theorem}
 If $p$ is invertible in the ring $R$, then the ring homomorphism $\omega: W(R) \rightarrow R^{\mathbb N_0}$ induced by the ghost map \eqref{ghostmap} is an isomorphism, i.e., $W(R) \cong R^{\mathbb N_0}$. It is obvious that the polynomials $S_n$ and $M_n$ are determined by the first $n+1$ coordinates of the Witt vectors and their coefficients do not depend upon the ring $R$. In particular, one can calculate
\begin{align*}
  S_0 &=x_0+y_0, \quad  S_1=x_1+y_1-\sum_{i=1}^{p-1}\frac{1}{p}{p \choose i}x_0^iy_0^{p-i},\\
  M_0 &=x_0y_0, \quad  M_1=x_0^py_1+x_1y_0^p+px_1y_1.
\end{align*}

\subsection{Polynomials $S_n$ and $M_n$ for $r$-fold operation}
The calculations of $S_n$ and $M_n$ for big $n$ are very complicated. However, for the purpose of this paper, we are more concerned with the degree of the polynomial $S_n$ for $r$-fold addition. We also give the degree of the polynomial $M_n$ for $r$-fold multiplication for completeness.

Let $r\in \N$. For $j\in [1,r]$, write $X_j=(x_{0j},\dots,x_{nj},\dots)$, and
\begin{equation}\label{sm}
(S_{0}^{(r)},\dots,S_{n}^{(r)},\dots)=X_1\oplus\cdots\oplus X_r, \quad (M_{0}^{(r)},\dots,M_{n}^{(r)},\dots)=X_1\odot\cdots\odot X_r.
\end{equation}
\begin{lemma}\label{homlem}
  Both $S_{n}^{(r)}$ and $M_{n}^{(r)}$ are polynomials with integer coefficients in $(n+1)r$ variables $x_{ij}(i\in [0,n], j\in [1,r])$. If we set $\mathrm{wt}(x_{ij})=p^i$ for $i\in [0,n]$ and $j\in [1,r]$, then $S_{n}^{(r)}$ is weighted homogeneous of weighted degree $p^n$, and $M_{n}^{(r)}$ is weighted homogeneous of weighted degree $rp^n$. More generally, let $d\in \N$, if we set $\mathrm{wt}(x_{ij})\leq dp^i$ for $i\in [0,n]$ and $j\in [1,r]$, then $S_{n}^{(r)}$ is of weighted degree $\leq dp^n$, and $M_{n}^{(r)}$ is of weighted degree $\leq rdp^n$.
\end{lemma}
\begin{proof}
 It immediately follows from \eqref{sm} and Theorem \ref{wittthm} that the polynomial $S_{n}^{(r)}$ has integer coefficients and that
\begin{equation*}
  \omega_n(S_{0}^{(r)},\dots,S_{n}^{(r)})=\omega_n(X_1)+\cdots+\omega_n(X_r),
\end{equation*}
which in expansion by \eqref{wittpoly} is
\begin{equation}\label{epp}
  (S_{0}^{(r)})^{p^n}+p(S_{1}^{(r)})^{p^{n-1}}+\cdots+p^n(S_{n}^{(r)})=(x_{01}^{p^n}+\cdots+x_{0r}^{p^n})+p(x_{11}^{p^{n-1}}+\cdots+x_{1r}^{p^{n-1}})+p^n(x_{n1}+\cdots+x_{nr}).
\end{equation}
Thus we have
\begin{equation}\label{ss}
  S_{n}^{(r)}=\frac{1}{p^n}\left(\sum_{i=1}^{r}\omega_n(X_i)-\sum_{i=0}^{n-1}p^i(S_{i}^{(r)})^{p^{n-i}}\right).
\end{equation}
  Since $S_{n}^{(r)}$ has integer coefficients, the factor $\frac{1}{p^n}$ in \eqref{ss} will be cancelled eventually. Set $\mathrm{wt}(x_{ij})=p^i$ for $i\in [0,n]$, $j\in [1,r]$. We make use of induction on $n$ to show that $S_{n}^{(r)}$ is a weighted homogeneous polynomial of weighted degree $p^n$ in $(n+1)r$ variables $x_{ij}(i\in [0,n], j\in [1,r])$. The case of $n=0$ in which $S_{0}^{(r)}=x_{01}+\cdots+x_{0r}$ is trivially verified. We assume that $S_{k}^{(r)}$ is a weighted homogeneous polynomial of weighted degree $p^k$ in $(k+1)r$ variables $x_{ij}(i\in [0,k], j\in [1,r])$ for $0\leq k\leq n-1$. Then the sum $\sum_{i=0}^{n-1}p^i(S_{i}^{(r)})^{p^{n-i}}$ in \eqref{ss} is a weighted homogeneous polynomial of weighted degree $p^{n}$ in $nr$ variables $x_{ij}(i\in [0,n-1], j\in [1,r])$. Note that by \eqref{epp} the sum $\sum_{i=1}^{r}\omega_n(X_i)$ in \eqref{ss} is weighted homogeneous of weighted degree $p^{n}$ in $(n+1)r$ variables $x_{ij}(i\in [0,n], j\in [1,r])$ with the variables $x_{nj}(j\in [1,r])$ not occurring in $\sum_{i=0}^{n-1}p^i(S_{i}^{(r)})^{p^{n-i}}$, which implies that $S_{n}^{(r)}\neq 0$. Thus we conclude that $S_{n}^{(r)}$ is a weighted homogeneous polynomial of weighted degree $p^n$ in $(n+1)r$ variables $x_{ij}(i\in [0,n], j\in [1,r])$. The other results can be similarly deduced.
\end{proof}
\begin{remark}\label{snmn}
  \textup{(i)} The weighted degree of $S_{n}^{(r)}$ does not depend upon $r$, but $M_{n}^{(r)}$ does.

  \textup{(ii)} To indicate explicitly the variables as well as their order in $S_n^{(r)}$ and $M_n^{(r)}$, we write
  \begin{align*}
  S_n^{(r)}:& =S_n^{(r)}(x_{01},\dots,x_{0r};x_{11},\dots,x_{1r};\dots;x_{n1},\dots,x_{nr}), \mbox{  and  } \\
  M_n^{(r)}:& =M_n^{(r)}(x_{01},\dots,x_{0r};x_{11},\dots,x_{1r};\dots;x_{n1},\dots,x_{nr}).
    \end{align*}

   \textup{(iii)}  For later applications, we record the explicit formulae for $S_0^{(r)}$ and $S_1^{(r)}$.
   \begin{align*}   S_0^{(r)} &= x_{01}+\cdots + x_{0r},\\
   S_1^{(r)} &= x_{11}+\cdots + x_{1r} - \frac{1}{p}\sum_{t_1+\cdots + t_r=p, \ 0\leq t_i\leq p-1} {p\choose {t_1, \cdots, t_r}}x_{01}^{t_1}\cdots
   x_{0r}^{t_r}.
 \end{align*}
\end{remark}

\subsection{Perfect rings with characteristic $p$}
In this subsection, we always let $R$ be a perfect ring with characteristic $p$, which means that the Frobenius map $\phi: a\mapsto a^p$ is an automorphism. Let $W(R)$ denote the ring of Witt vectors over $R$. The Teichm\"uller lifting is defined by
  \begin{equation*}
     \tau: R \hookrightarrow W(R), \quad  a \mapsto \tau(a)=(a,0,0,\dots),
  \end{equation*}
and $\tau(a)$ is called the Teichm\"uller representative of the element $a$. Let
  \begin{equation*}
     K_R:=\{\tau(a_0)+\tau(a_1)p+\tau(a_2)p^2+\cdots \mid a_i\in R, i=0,1,2,\dots\}.
  \end{equation*}
Then $K_R$ will be a $p$-adic ring under the usual addition and multiplication via its isomorphism with $W(R)$.  Moreover, if $R$ is a field, then $K_R$ is a complete discrete valuation ring of zero characteristic with residue field $R$ and maximal ideal $pK_R$. Each element $\mathbf (a_0,a_1,a_2\dots)\in W(R)$ can be uniquely represented in $K_R$ as
\begin{equation*}
  \tau(a_0)+\tau(a_1)p+\tau(a_2)p^2+\cdots
\end{equation*}
However, this bijection is not a ring isomorphism between $W(R)$ and $K_R$ because it does not respect the addition. Since $R$ is a perfect ring with characteristic $p$, we have $R\cong R^p$ via the Frobenius map $\phi: a\mapsto a^p$. The true ring isomorphism between $W(R)$ and $K_R$ is denoted by $\tau$ again and given explicitly by
\begin{equation*}
  \tau: W(R) \rightarrow K_R,\quad  (a_0,a_1,a_2,\dots)\mapsto \tau(a_0)+\tau(a_1)^{p^{-1}}p+\tau(a_2)^{p^{-2}}p^2+\cdots
\end{equation*}
We usually adopt the alternative expression for $\tau$ given as below
\begin{equation}\label{newmap}
  \tau: W(R) \rightarrow K_R,\quad  (a_0,a_1^p,a_2^{p^2},\dots)\mapsto \tau(a_0)+\tau(a_1)p+\tau(a_2)p^2+\cdots
\end{equation}
The advantage of the expression \eqref{newmap} lies in that it makes the Witt polynomials become homogenous (cf. Remark \ref{homfact}) and hence the polynomial $S_{n}^{(r)}$ is weighted homogeneous of weighted degree $p^n$ by Lemma \ref{homlem}.
\begin{example}
  A well-known example is that $W(\F_q)\cong\Z_q$. In particular, the finite Witt ring $W(\F_q)/p^mW(\F_q)$
  becomes $\Z_q/p^m\Z_q$. We will discuss this example in detail in the next section.
\end{example}

Now let $r\in \N$, we dicuss the $r$-fold addition and the $r$-fold multiplication in $K_R$. For $j\in [1,r]$, let
$$X_j=(x_{0j},x_{1j}^p,x_{2j}^{p^2},\cdots)\in W(R).$$
Then,
$$\tau(X_j)=\sum_{i}^{\infty}\tau(x_{ij})p^i\in K_R,$$where $\tau$ denotes the ring isomorphism between $W(R)$ and $K_R$ given by \eqref{newmap}. We want to find two functions $\widetilde{s}_n^{(r)}$ and $\widetilde{m}_n^{(r)}$, which behave like $S_n^{(r)}$ and $M_n^{(r)}$ as defined in Lemma \ref{homlem}, such that
\begin{equation}\label{smallsm}
\sum_{j=1}^{r}\left(\sum_{i=0}^{\infty}\tau(x_{ij})p^i\right)=\sum_{n=0}^{\infty}\tau(\widetilde{s}_n^{(r)})p^n, \quad \mbox{     and     }  \quad \prod_{j=1}^{r}\left(\sum_{i=0}^{\infty}\tau(x_{ij})p^i\right)=\prod_{n=0}^{\infty}\tau(\widetilde{m}_n^{(r)})p^n.
\end{equation}

The formulae for $\widetilde{s}_n^{(r)}$ and $\widetilde{m}_n^{(r)}$ are given below, and the proof for $r=2$ can also be found in \cite[Theorem 1.5]{rabinoff}.

\begin{theorem}\label{smm} With the above notation, for $n\in \N_0$ we have
  \begin{align*}
  \widetilde{s}_n^{(r)}& =S_n^{(r)}(x_{01}^{1/{p^n}},\dots,x_{0r}^{1/{p^n}};x_{11}^{1/{p^{n-1}}},\dots,x_{1r}^{1/{p^{n-1}}};\dots;x_{n1},\dots,x_{nr}), \mbox{  and  } \\
  \widetilde{m}_n^{(r)}& =M_n^{(r)}(x_{01}^{1/{p^n}},\dots,x_{0r}^{1/{p^n}};x_{11}^{1/{p^{n-1}}},\dots,x_{1r}^{1/{p^{n-1}}};\dots;x_{n1},\dots,x_{nr}).
  \end{align*}
  Let $s_n^{(r)}:=(\widetilde{s}_n^{(r)})^{p^n}$ and $m_n^{(r)}:=(\widetilde{m}_n^{(r)})^{p^n}$, then we have
    \begin{align*}
s_n^{(r)}& =S_n^{(r)}(x_{01},\dots,x_{0r};x_{11}^{p},\dots,x_{1r}^{p};\dots;x_{n1}^{p^n},\dots,x_{nr}^{p^n}), \mbox{  and  } \\
m_n^{(r)}& =M_n^{(r)}(x_{01},\dots,x_{0r};x_{11}^{p},\dots,x_{1r}^{p};\dots;x_{n1}^{p^n},\dots,x_{nr}^{p^n}).
  \end{align*}
  The polynomials $s_n^{(r)}$ and $m_n^{(r)}$ have integer integers. Moreover, $s_n^{(r)}$ is homogeneous of degree $p^n$ and $m_n^{(r)}$ is homogeneous of degree $rp^n$ in the variables $x_{ij}$.
\end{theorem}
\begin{proof}
  We only consider the $r$-fold addition. Let $\oplus$ denote the addition in the ring $W(R)$. Like in \eqref{sm}, we have
\begin{equation}\label{sm22}
(S_{0}^{[r]},\dots,S_{n}^{[r]},\dots)=X_1\oplus\cdots\oplus X_r,
\end{equation}
where
$$S_n^{[r]} = S_n^{(r)}(x_{01},\dots,x_{0r};x_{11}^{p},\dots,x_{1r}^{p};\dots;x_{n1}^{p^n},\dots,x_{nr}^{p^n}).$$
 Applying the map $\tau$ to the two sides of \eqref{sm22} and combining \eqref{smallsm} yields
\begin{equation*}
\sum_{n=0}^{\infty}\tau(S_{n}^{[r]})^{p^{-n}}p^n=\sum_{j=1}^{r}\tau(X_j)=\sum_{j=1}^{r}\left(\sum_{i=0}^{\infty}\tau(x_{ij})p^i\right)=\sum_{n=0}^{\infty}\tau(\widetilde{s}_n^{(r)})p^n.
\end{equation*}
 Therefore $\tau(S_{n}^{[r]})^{p^{-n}}=\tau(\widetilde{s}_n^{(r)})$ and hence $(S_{n}^{[r]})^{p^{-n}}=\widetilde{s}_n^{(r)}$ for all $n\in \N_0$. That is,
   \begin{align}\label{6677}
  \widetilde{s}_n^{(r)}& =(S_n^{(r)})^{p^{-n}}(x_{01},\dots,x_{0r};x_{11}^{p},\dots,x_{1r}^{p};\dots;x_{n1}^{p^n},\dots,x_{nr}^{p^n}).
  \end{align}
 Since $R$ is perfect with characteristic $p$, we can put the power $p^{-n}$ inside, namely
 \begin{align*}
  \widetilde{s}_n^{(r)}& =S_n^{(r)}(x_{01}^{1/{p^n}},\dots,x_{0r}^{1/{p^n}};x_{11}^{1/{p^{n-1}}},\dots,x_{1r}^{1/{p^{n-1}}};\dots;x_{n1},\dots,x_{nr}).
  \end{align*}
 Note the degrees of variables in $\widetilde{s}_n^{(r)}$ are fractions. To apply the Ax-Katz theorem later, we need them to be integers. Let $s_n^{(r)}:=(\widetilde{s}_n^{(r)})^{p^n}$. Then by \eqref{6677}, we have
   \begin{align}\label{jjkkk}
  s_n^{(r)}& =S_n^{(r)}(x_{01},\dots,x_{0r};x_{11}^{p},\dots,x_{1r}^{p};\dots;x_{n1}^{p^n},\dots,x_{nr}^{p^n}).
  \end{align}
  It follows from Theorem \ref{wittpoly} that the polynomials $\widetilde{s}_n^{(r)}$ and  $s_n^{(r)}$ have integer coefficients and from Lemma \ref{homlem} that $s_n^{(r)}$ is homogeneous of degree $p^n$. The results for $\widetilde{m}_n^{(r)}$ and $m_n^{(r)}$ can be deduced similarly.
\end{proof}
\begin{remark}
   In the following we simply write $s_n^{(r)}$, which means by default it is in the variables $\{x_{ij}^{p^i} \mid i\in[0,n],j\in[1,r]\}$. As presented in \eqref{jjkkk} the order of $x_{ij}^{p^i}$ may affect the expression of $s_n^{(r)}$, but it does not affect the homogeneous degree of $s_n^{(r)}$, with which we are most concerned. So we may loosely write  $s_n^{(r)}=s_n^{(r)}(x_{ij}^{p^i} \mid i\in[0,n],j\in[1,r])$ when the variables are needed to be indicated.
\end{remark}

\begin{lemma}\label{zerolem}
  Let $R$ be a perfect ring with characteristic $p$. Let $m\in \N$ and $\sum_{i=0}^{\infty}\tau(x_i)p^i\in K_R$ with $x_i\in R$. Then the following statements are equivalent:
    \begin{enumerate}
    \item[\textup{(i)}]$\sum_{i=0}^{\infty}\tau(x_i)p^i\equiv 0 \pmod {p^m}$.
    \item[\textup{(ii)}]$x_0=x_1=\dots=x_{m-1}=0$.
    \item[\textup{(iii)}]$x_0=x_1^{p}=\dots=x_{m-1}^{p^{m-1}}=0$.
  \end{enumerate}
\end{lemma}
\begin{proof}
It follows directly from that $\tau((0,0,\dots))=0$ and that $\tau$ is a ring isomorphism between $W(R)$ and $K_R$.
\end{proof}

The above lemma can be easily extended to the $r$-fold addition, which will play the crucial role in our later proofs.
\begin{lemma}\label{keylem}
  Let $R$ be a perfect ring with characteristic $p$. Let $m,r\in \N$ and $\sum_{i=0}^{\infty}\tau(x_{ij})p^i\in K_R$ with $x_{ij}\in R, j\in [1,r]$. Suppose $\sum_{j=1}^{r}\left(\sum_{i=0}^{\infty}\tau(x_{ij})p^i\right)=\sum_{n=0}^{\infty}\tau(\widetilde{s}_n^{(r)})p^n$. For $n\in \N_0$, let $s_n^{(r)}:=(\widetilde{s}_n^{(r)})^{p^n}$. Then the following statements are equivalent:
    \begin{enumerate}
    \item[\textup{(i)}]$\sum_{j=1}^{r}\sum_{i=0}^{\infty}\tau(x_{ij})p^i\equiv 0 \pmod {p^m}$.
    \item[\textup{(ii)}]$\widetilde{s}_0^{(r)}=\widetilde{s}_1^{(r)}=\dots=\widetilde{s}_{m-1}^{(r)}=0$.
    \item[\textup{(iii)}]$s_0^{(r)}=s_1^{(r)}=\dots=s_{m-1}^{(r)}=0$.
    \end{enumerate}
\end{lemma}

\section{$q$-Divisibility theorem for the Teichm\"{u}ller box}\label{sect3}

Let $p$ be a prime number and $q=p^h$ with $h\in \N$. In this section, we always let $R=\F_q$ and denote by $W(\F_q)$ the ring of Witt vectors over $\F_q$. Since we study the polynomials over the ring $\Z_q$, we introduce a bit more pertinent facts about $\Z_q$.

Let $\Q_p$ denote the field of $p$-adic rational numbers and $\Z_p$ the ring of integers in $\Q_p$. Let $\Q_q$ denote the unramified extension of $\Q_p$ of degree $h$ and $\Z_q$ the ring of integers in $\Q_q$. The residue field of $\Z_q$ is $\F_q$, i.e., $\Z_q/p\Z_p\cong \F_q$, and the quotient field of $\Z_q$ is $\Q_q$. Suppose that $\mu_{q-1}$ is a primitive $(q-1)$-th root of unity in $\Q_q$. Then $\Q_q=\Q_p[\mu_{q-1}]$ and $\Z_q=\Z_p[\mu_{q-1}]$.

Let $T_q$ be the set of Teichm\"{u}ller representatives of $\mathbb{F}_{q}$ in $\mathbb{Z}_q$ and the related Teichm\"{u}ller lifting be $\tau:\mathbb{F}_{q}\rightarrow \Z_q, a\mapsto \tau(a)$. Then $T_q=\{\tau(a) | a\in \F_q\}$ and for each $a\in \F_q$, we have $\tau(a)^q=\tau(a)$ and $\tau(a)\equiv a \pmod{p}$. Then $T_q=\{\mu_{q-1}^i | i=1,2,\dots,q-1\}\cup \{0\}$. For any $a\in T_q$, let $\widetilde{a}$ be the unique element in $\F_q$ such that $\tau(\widetilde{a})=a$. We call $T_q^n$ the Teichm\"{u}ller box in $\Z_q^n$.

Another construction of $\Z_q$ is using the ring $W(\F_q)$, the Witt vectors over $\F_q$, as described in Section \ref{sect2}. The ring isomorphism between $W(\F_q)$ and $\Z_q$ is given by
\begin{equation}\label{zqmap}
  \tau:W(\F_q) \rightarrow \Z_q,\quad  (a_0,a_1^p,a_2^{p^2},\dots)\mapsto \tau(a_0)+\tau(a_1)p+\tau(a_2)p^2+\cdots
\end{equation}
If $\F_q=\F_p$, then $W(\F_p)\cong\Z_p$. Moreover, since $a^p=a$ for $a\in\F_p$, \eqref{zqmap} becomes
\begin{equation}\label{zpmap}
  \tau:W(\F_p) \rightarrow \Z_p,\quad  (a_0,a_1,a_2,\dots)\mapsto \tau(a_0)+\tau(a_1)p+\tau(a_2)p^2+\cdots
\end{equation}

Let $\mathbb{Z}_q[x_1,\ldots,x_n]$ denote the ring of polynomials in $n$ variables $x_1,\ldots,x_n$ with coefficients in $\mathbb{Z}_q$. Write $X^u=x_1^{d_1}\cdots x_n^{d_n}$ with $u=(d_1,\dots,d_n)\in \N_0^n$. Let $f=\sum_{j=1}^{r}a_jX^{u_j}\in \mathbb{Z}_q[x_1,\ldots,x_n]$ with $0\neq a_j\in \mathbb{Z}_q$.
We can write
$$a_j=\sum_{i=0}^{\infty}a_{ij}p^i, \ \ a_{ij}\in T_q.$$
This is called the Teichm\"uller expansion of $a_j$. Similarly,
$$f=\sum_{j=1}^{r}\sum_{i=0}^{\infty}a_{ij}p^iX^{u_j}=\sum_{j=1}^{r}\sum_{i=0}^{\infty}(a_{ij}X^{u_j})p^i
=\sum_{i=0}^{\infty}p^i \sum_{j=1}^{r}a_{ij}X^{u_j}$$
is the Teichmuller expansion of the polynomial $f$. We first consider the single polynomial case.

\subsection{For single polynomial}

A polynomial $f\in\Z_q[x_1,\dots, x_n]$ is called a Teichm\"uller polynomial if all of its coefficients are
Teichm\"uller elements in $T_q$. Clearly, any polynomial $f\in\Z_q[x_1,\dots, x_n]$ has the unique expansion
$$f(X) = \sum_{i=0}^{\infty} p^i f_i(X),$$
where each $f_i(X)$ is a Teichm\"uller polynomial. This is called the Teichm\"uller expansion of $f$.
It is obtained from the Teichm\"uller expansion of the coefficients of $f$.

\begin{theorem}[Strong Version]\label{thm1}
Let $p$ be a prime number and $q=p^h$ with $h\in \N$. Let $f\in \mathbb{Z}_q[x_1,\ldots,x_n]$ be a nonzero polynomial. Given an $m\in \N$, let
\begin{equation*}
  V:=\{X\in T_q^n \mid  f(X)\equiv 0 \pmod {p^{m}}\}.
\end{equation*}
Let $f=\sum_{i=0}^{\infty} p^i f_i$ be the Teichm\"uller expansion of $f$. Let $d\in \N$. If $\deg (f_i)\leq dp^{h\lfloor\frac{i}{h}\rfloor}$ for all $i\in [0,m-1]$, then
\begin{equation}\label{hhgg1}
\mathrm{ord}_q(|V|)\geq \left\lceil\frac{n-\frac{p^{m}-1}{p-1}d}{p^{m-1}d}\right\rceil^*.
\end{equation}
\end{theorem}
\begin{proof}
  Let $s_n^{(r)}$ be the polynomial as defined before, which is homogeneous of degree $p^n$ by Theorem \ref{smm}. Write
  $$f= \sum_{i=0}^{\infty} p^i f_i= \sum_{i=0}^{\infty}p^i\sum_{j=1}^{r}a_{ij}X^{u_j}, \ a_{ij}\in T_q.$$
  From Lemma \ref{keylem} we know that for a given $X\in T_q^n$, $f(X)\equiv 0 \pmod {p^m}$ if and only if
  $$g_k(X):=s_k^{(r)}\left((\widetilde{a}_{ij}\widetilde{X}^{u_j})^{p^i}\mid i\in[0,k], j\in [1,r]\right)=0,  \mbox{ for all } k\in [0,m-1].$$
  Note $\widetilde{a}_{ij}, \widetilde{X}\in \F_q$ with $X\in T_q$. Define
  \begin{equation*}
  \widetilde{V}:=\left\{X\in \F_q^n \mid  g_k(X):=s_k^{(r)}\left((\widetilde{a}_{ij}X^{u_j})^{p^i}\mid i\in[0,k], j\in [1,r]\right)= 0 \mbox{ for all } k\in [0,m-1]\right\}.
  \end{equation*}
   Then $|V|=|\widetilde{V}|$. Note $(\widetilde{a}_{ij}X^{u_j})^{p^i}=(\widetilde{a}_{ij}X^{u_j})^{p^{i-h\lfloor\frac{i}{h}\rfloor}}$ in $\F_q$ with $q=p^h$. Now
   $$\deg (a_{ij}X^{u_j}) \leq \deg(f_i) \leq dp^{h\lfloor\frac{i}{h}\rfloor}.$$
   It follows that
   $$\deg((\widetilde{a}_{ij}X^{u_j})^{p^{i-h\lfloor\frac{i}{h}\rfloor}})\leq dp^{h\lfloor\frac{i}{h}\rfloor} p^{i-h\lfloor\frac{i}{h}\rfloor}= dp^i.$$
   So by Lemma \ref{homlem},  $\deg(g_k)\leq d p^k$ for $k\in [0,m-1]$. Thus
\begin{equation}\label{ddee}
    \sum\nolimits_{k=0}^{m-1}\deg(g_k)\leq \sum\nolimits_{k=0}^{m-1}p^kd=\frac{p^m-1}{p-1}d.
\end{equation}
  Applying the Ax-Katz theorem \ref{axkatzthm} to $\widetilde{V}$ and using \eqref{ddee}, we obtain
  \begin{equation}\label{lleq}
  \mathrm{ord}_q(|\widetilde{V}|)\geq \left\lceil\frac{n-\sum_{k=0}^{m-1}\deg(g_k)}{\max\nolimits_{k\in [0,m-1]}\deg(g_k)}\right\rceil^*\geq \left\lceil\frac{n-\frac{p^{m}-1}{p-1}d}{p^{m-1}d}\right\rceil^*.
  \end{equation}
  Then \eqref{hhgg1} follows from \eqref{lleq} and the equality that $|V|=|\widetilde{V}|$.
\end{proof}

If $\deg(f) =d$, then $\deg (f_i)\leq d \leq dp^{h\lfloor\frac{i}{h}\rfloor}$ for all $i$ and thus the condition of the theorem
is automatically satisfied. This gives the following weaker but simpler consequence.

\begin{corollary}[Weak Version]\label{cor1}
Let $p$ be a prime number and $q=p^h$ with $h\in \N$. Let $f\in \mathbb{Z}_q[x_1,\ldots,x_n]$ be a nonzero polynomial. Given an $m\in \N$, let
\begin{equation*}
  V:=\{X\in T_q^n \mid  f(X)\equiv 0 \pmod {p^{m}}\}.
\end{equation*}
Then
\begin{equation*}
\mathrm{ord}_q(|V|)\geq \left\lceil\frac{n-\frac{p^{m}-1}{p-1}\deg (f)}{p^{m-1}\deg (f)}\right\rceil^*.
\end{equation*}
\end{corollary}
Corollary \ref{cor1}, though in which the condition is weaker than Theorem \ref{thm1}, can be viewed as the generalized $\Z_q$-version of Theorem \ref{grythm} with $s=1$ for the Teichm\"{u}ller box case. In other words, Theorem \ref{thm1} not only generalizes but also improves Theorem \ref{grythm} for one polynomial in the Teichm\"{u}ller box case.

Corollary \ref{cor2} below follows from Theorem \ref{thm1} and the fact that $h\lfloor\frac{i}{h}\rfloor=i$ for $h=1$.
\begin{corollary}\label{cor2}
Let $p$ be a prime number. Let $f\in \mathbb{Z}_p[x_1,\ldots,x_n]$ be a nonzero polynomial. Given an $m\in \N$, let
\begin{equation*}
  V:=\{X\in T_p^n \mid  f(X)\equiv 0 \pmod {p^{m}}\}.
\end{equation*}
Let $f=\sum_{i=0}^{\infty} p^i f_i$ be the Teichm\"uller expansion of $f$. Let $d\in \N$. If $\deg (f_i)\leq dp^i$ for all $i\in [0,m-1]$, then
\begin{equation*}
\mathrm{ord}_p(|V|)\geq \left\lceil\frac{n-\frac{p^{m}-1}{p-1}d}{p^{m-1}d}\right\rceil^*.
\end{equation*}
\end{corollary}

Note that the degree condition $\deg (f_i)\leq dp^{i}$ is significantly weaker than the condition $\deg(f) \leq d$, which allows those terms of $f$ that are divisible by $p$ have much larger degree than $d$.

\subsection{For polynomial system}

Theorem \ref{thm1} can be extended to the system of polynomials without much more difficulties except for more cumbersome notation. Theorem \ref{thm2} below generalizes as well as improves Theorem \ref{grythm} for the Teichm\"{u}ller box case.

\begin{theorem}[Strong Version]\label{thm2}
Let $p$ be a prime number and $q=p^h$ with $h\in \N$. Let $f_1,\dots, f_s\in \mathbb{Z}_q[x_1,\ldots,x_n]$ be a system of nonzero polynomials. For given $m_1,\dots,m_s\in \N$, let
\begin{equation*}
  V:=\{X\in T_q^n \mid  f_k(X)\equiv 0 \pmod {p^{m_k}} \mbox{ for all } k\in [1,s]\}.
\end{equation*}
Write the $p$-adic Teichm\"uller expansion
$$f_k=\sum_{i=0}^{\infty} p^if_{k,i}(X),  \ k\in[1,s].$$
Let $d_1,\dots,d_s\in \N$. If $\deg (f_{k,i})\leq d_kp^{h\lfloor\frac{i}{h}\rfloor}$ for all $i\in [0,m_k-1]$, $k\in[1,s]$, then
\begin{equation*}
\mathrm{ord}_q(|V|)\geq \left\lceil\frac{n-\sum_{k=1}^s\frac{p^{m_k}-1}{p-1}d_k}{\max\nolimits_{k \in [1,s]}\{p^{m_k-1}d_k\}}\right\rceil^*.
\end{equation*}
\end{theorem}
\begin{proof}
   From the proof of Theorem \ref{thm1}, we see that for each modulus $p^{m_k}$, the polynomial $f_k$ contributes $m_k$ polynomials $g_{tk}$ over $\F_q$ ($t\in [0,{m_k}-1]$)  whose degree is bounded by $p^{t}d_k$ and thus
  $\sum_{t=0}^{{m_k}-1}\deg(g_{tk})\leq \frac{p^{m_k}-1}{p-1}d_k$. Now given $s$ polynomials $f_k$ and $s$ moduli $p^{m_k}$, $k\in[1,s]$, we get $\sum_{k=1}^{s}m_k$ polynomials over $\F_q$ with the sum of degrees $\leq \sum_{k=1}^{s}\frac{p^{m_k}-1}{p-1}d_k$ and maximal degree bounded by $\max\nolimits_{k \in [1,s]}\{p^{m_k-1}d_k\}$. Applying the Ax-Katz theorem, we obtain the desired result.
\end{proof}

Let $d_k =\deg(f_k)$. Then trivially we have $\deg (f_{k,i}) \leq d_k \leq d_kp^{h\lfloor\frac{i}{h}\rfloor}$.
This gives the following weaker corollary.

\begin{corollary}[Weak Version]\label{cor3}
Let $p$ be a prime number and $q=p^h$ with $h\in \N$. Let $f_1,\dots, f_s\in \mathbb{Z}_q[x_1,\ldots,x_n]$ be a system of nonzero polynomials. For given $m_1,\dots,m_s\in \N$, let
\begin{equation*}
  V:=\{X\in T_q^n \mid  f_k(X)\equiv 0 \pmod {p^{m_k}} \mbox{ for all } k\in [1,s]\}.
\end{equation*}
Then
\begin{equation*}
\mathrm{ord}_q(|V|)\geq \left\lceil\frac{n-\sum_{k=1}^s\frac{p^{m_k}-1}{p-1}\deg (f_k)}{\max\nolimits_{k \in [1,s]}\{p^{m_k-1}\deg (f_k)\}}\right\rceil^*.
\end{equation*}
\end{corollary}

In the case $q=p$, the theorem becomes

\begin{corollary}\label{cor4}
 Let $p$ be a prime number. Let $f_1,\dots, f_s\in \mathbb{Z}_p[x_1,\ldots,x_n]$ be a system of nonzero polynomials. For given $m_1,\dots,m_s\in \N$, let
\begin{equation*}
  V:=\{X\in T_p^n \mid  f_k(X)\equiv 0 \pmod {p^{m_k}} \mbox{ for all } k\in [1,s]\}.
\end{equation*}
For each $k\in[1,s]$, write the $p$-adic Teichm\"uller expansion
$$f_k=\sum_{i=0}^{\infty} p^if_{k,i}(X).$$
Let $d_1,\dots,d_s\in \N$. If $\deg (f_{k,i})\leq d_kp^i$ for all $i\in [0,m_k-1]$, $k\in[1,s]$, then
\begin{equation*}
\mathrm{ord}_p(|V|)\geq \left\lceil\frac{n-\sum_{k=1}^s\frac{p^{m_k}-1}{p-1}d_k}{\max\nolimits_{k \in [1,s]}\{p^{m_k-1}d_k\}}\right\rceil^*.
\end{equation*}
\end{corollary}

Note that in the $p$-adic expansion of the polynomial $f_k$, the condition $\deg (f_{k,i})\leq d_kp^i$
for $i\geq 1$ is significantly weaker than the condition $\deg(f_k) \leq d_k$. Namely, the degree of those
terms in $f_k$ which are divisible by $p$ can have much larger degree than $d_k$.


\section{$q$-divisibility theorem for a general box}\label{sect4}



The box $T_q^n$ in the previous section is called the Teichm\"{u}ller box. A natural question is whether our results in Section \ref{sect3}, especially Theorem \ref{thm2}, hold true for other non-Teichm\"{u}ller boxes $\mathcal B$.
We address this question in this section. For this purpose, we first need to understand a general box algebraically.

Recall that a box $\mathcal B$ in $\Z_q^n$ is defined to be a complete system of representatives of $\F_q^n$ in $\Z_q^n$. The box $\mathcal B$ considered in Theorem \ref{grythm} is a special case in split form, that is, $\mathcal B =\mathcal I_1\times \dots\times \mathcal I_n$, where each $\mathcal I_i$ is a complete system of representatives of $\F_q$ in $\Z_q/p \Z_q$.
Now, we would like to describe the box $\mathcal B$ in terms of the image of a polynomial system.

A polynomial $g\in\Z_q[x_1,\dots, x_n]$ is called a Teichm\"uller polynomial if all of its coeffcients are
Teichm\"uller elements in $T_q$. The polynomial $g$ is called reduced if its degree in each variable is at most $q-1$.
Thus, a reduced polynomial $g\in\Z_q[x_1,\dots, x_n]$ has total degree at most $n(q-1)$.
For any given box $\mathcal B$, the elements in $\mathcal B$ can be uniquely determined by a system of reduced
polynomials over $\Z_q$.

\begin{lemma}\label{lemkkgg}
 Let $p$ be a prime number and $q=p^h$ with $h\in \N$. Let $\mathcal B\subseteq\Z_q^n$ with $|\mathcal B|=q^n$ and $\mathcal B \mod p=\F_q^n$.
     \begin{itemize}
    \item[(i)]There exists a unique system of reduced Teichm\"uller polynomials $g_{ij}\in\Z_q[x_1,\dots, x_n]$ depending only on
 the box $\mathcal B$
 with $j\in [1,n],i\in\N$ such that for any $Y=(y_1, \dots, y_n)\in \mathcal B$, we have
\begin{equation}\label{zz}
     Y=X + (g_{11}(X),\dots,g_{1n}(X))p+ (g_{21}(X),\dots,g_{2n}(X))p^2+\cdots,
\end{equation}
where $X=(x_1, \dots, x_n)\in T_q^n$ is the Teichm\"uller lifting of the modulo $p$ reduction of $Y$.
    \item[(ii)]There exists a unique system of reduced polynomials $g_{j}\in\Z_q[x_1,\dots, x_n]$ depending only on
 the box $\mathcal B$
 with $j\in [1,n]$ such that for any $Y=(y_1, \dots, y_n)\in \mathcal B$, we have
\begin{equation}\label{zzz}
     Y=X + (g_{1}(X),\dots,g_{n}(X))p,
\end{equation}
where $X=(x_1, \dots, x_n)\in T_q^n$ is the Teichm\"uller lifting of the modulo $p$ reduction of $Y$. In particular, $\mathcal B$ is the image of $T_q^n$ under the polynomial map $X \rightarrow X + (g_{1}(X),\dots,g_{n}(X))p$.
    \item[(iii)]Assume that $\mathcal B$ is in split form. Then \eqref{zz} becomes
\begin{equation*}
     Y=X + (g_{11}(x_1),\dots,g_{1n}(x_n))p+ (g_{21}(x_1),\dots,g_{2n}(x_n))p^2+\cdots,
\end{equation*}
where each $g_{ij}(x_j)\in\Z_q[x_j]$ is a reduced Teichm\"uller polynomial in the one variable $x_j$ and hence has degree at most $q-1$ for $j\in [1,n],i\in\N$. Equivalently, \eqref{zzz} becomes
\begin{equation*}
     Y=X + (g_{1}(x_1),\dots,g_{n}(x_n))p,
\end{equation*}
where each $g_{j}(x_j)\in\Z_q[x_j]$ is a reduced polynomial in the one variable $x_j$ and hence degree at most $q-1$ for all $1\leq j\leq n$.
    \end{itemize}
\end{lemma}
\begin{proof} (ii) is equivalent to (i) by taking
$$g_j =\sum_{i=1}^{\infty} p^{i-1} g_{ij}, \ 1\leq j \leq n.$$
Thus, the right side is just the Teichm\"uller expansion of the left side. (iii) is a consequence of (i) and (ii) by applying them
to each one dimensional factor of the split box $\mathcal B$. We shall now prove (i).

For a given $Y=(y_1, \dots, y_n)\in \mathcal B$, we can write uniquely
\begin{equation*}
  (y_1, \dots, y_n)=(x_{01}, \dots, x_{0n})+(x_{11}, \dots, x_{1n})p,
\end{equation*}
where $X=(x_{01}, \dots, x_{0n}) \in T_q^n$ and $(x_{11}, \dots, x_{1n}) \in \Z_q^n$.
The vector $X=(x_{01}, \dots, x_{0n})$ in $T_q^n$ and the vector $Y=(y_1, \dots, y_n)$ in $\mathcal B$ determine each other.
In fact, $X$ is just the Teichm\"uller lifting of the reduction $Y \mod p$, and $Y$ is the unique element in $\mathcal B$ with the
same mod $p$ reduction as $X$. In particular, $(x_{11}, \dots, x_{1n})$ is also uniquely determined by
 $X=(x_{01}, \dots, x_{0n})$.

Letting $Y$ run over $\mathcal B$, then $X$ runs over $T_q^n$ as $\mathcal B \mod p=\F_q^n$ by assumption.
For each $1\leq j \leq n$, the quantity $x_{1j}$ is a function of $X$. This establishes a map from $T_q^n$ to $T_q$, and we consider the corresponding map from $\F_q^n$ to $\F_q$, namely
 \begin{equation*}
  \widetilde{g}_{1j}:\quad \F_q^n \rightarrow \F_q,\quad  X\mapsto \widetilde{g}_{1j}(X).
\end{equation*}
Recall the fact that any map from $\F_q^n$ to $\F_q$ can be expressed uniquely by a reduced polynomial in $n$ variables with coefficients in $\F_q$. In particular, our map $\widetilde{g}_{1j}$ is a reduced polynomial in $\mathbb{F}_q[x_1,\cdots, x_n]$.
Let $g_{1j}$ be the Teichm\"uller lifting of $\widetilde{g}_{1j}$ in $\Z_q[x_1,\dots,x_n]$. Then, we have proved
$$Y = X + (g_{11}(X), \cdots, g_{1n}(X))p + (x_{21}, \cdots, x_{2n})p^2,$$
where $(x_{21}, \cdots, x_{2n}) \in \Z_q^n$ is uniquely determined by $X$. Continuing this procedure, we find
uniquely determined reduced Teichm\"uller polynomials $g_{ij}(X) \in \Z_q[x_1, \cdots, x_n]$ ($i\geq 1$, $1\leq j\leq n$) such that (i) holds.
The lemma is proved.

 \end{proof}

\begin{remark}

For convenience, set $g_{0j}(X)=x_j$ for $j\in [1,n]$. Then, equation \eqref{zz} becomes
\begin{equation}
     Y= (g_{01}(X),\dots,g_{0n}(X))+ (g_{11}(X),\dots,g_{1n}(X))p+ (g_{21}(X),\dots,g_{2n}(X))p^2+\cdots
\end{equation}
We simply write $\mathcal B=T_q^n(g_{ij}:j\in [1,n],i\in\N)$ provided that $g_{ij}$'s are reduced Teichm\"uller polynomials
in  $\Z_q[x_1,\dots, x_n]$.
\end{remark}

This completes our discussion on the box $\mathcal B$. We now move to the reduction  from polynomial congruences in the
box $\mathcal B$ to equations over the finite field $\mathbb{F}_q$.

Fix a nonzero vector $\alpha=(\alpha_1,\dots,\alpha_n)\in \N_0^n$ with $|\alpha|:=\alpha_1+\dots+\alpha_n$. For an element $\beta\in \N_0^{|\alpha|}$ of the form
\begin{equation}\label{77nn}
  \beta=(\beta_{11},\dots,\beta_{\alpha_{1}1};\beta_{12},\dots,\beta_{\alpha_{2}2};\dots; \beta_{1n},\dots,\beta_{\alpha_{n}n}),
\end{equation}
we write $\beta=(\beta_{tl})$ for short with $\beta_{tl}$ arranged as in \eqref{77nn} and define $|\beta|:=\sum_{l=1}^{n}\sum_{t=1}^{\alpha_l}\beta_{tl}$. By \eqref{zz}, one has $Y=(y_1,\dots,y_n)$ with $y_j=\sum_{i=0}^{\infty}g_{ij}(X)p^i$ for $j\in[1,n]$. Then
\begin{equation}\label{77cc}
\prod_{l=1}^{n}\left(\sum_{k=0}^{\infty}g_{_{kl}}(X)p^k\right)^{\alpha_l}=
\sum\prod_{l=1}^{n}\prod_{t=1}^{\alpha_l}\left(g_{_{{\beta_{tl}}l}}(X)p^{\beta_{tl}}\right)=
\sum\left(\prod_{l=1}^{n}\prod_{t=1}^{\alpha_l}g_{_{{\beta_{tl}}l}}(X)\right)p^{|\beta|},
\end{equation}
where both the sums in \eqref{77cc} run over all the vectors $\beta=(\beta_{tl})\in \N_0^{|\alpha|}$. In particular, if $\deg(g_{ij})\leq
p^{h\lfloor\frac{i}{h}\rfloor}$ for $j\in [1,n],i\in [1, m-1]$, then for any $\beta=(\beta_{tl})\in \N_0^{|\alpha|}$ with $|\beta|\leq m-1$, we have
\begin{equation}\label{77aa}
\deg\left(\prod_{l=1}^{n}\prod_{t=1}^{\alpha_l}g_{_{{\beta_{tl}}l}}(X)\right)=\sum_{l=1}^{n}\sum_{t=1}^{\alpha_l}\deg(g_{_{{\beta_{tl}}l}}(X))\leq \sum_{l=1}^{n}\sum_{t=1}^{\alpha_l}p^{h\lfloor\frac{\beta_{tl}}{h}\rfloor}
\leq |\alpha|p^{h\lfloor\frac{|\beta|}{h}\rfloor}.
\end{equation}

\subsection{For single polynomial}

Like in Section \ref{sect3}, we first consider the single polynomial case.

\begin{theorem}[Strong Version]\label{thm3}
Let $p$ be a prime number and $q=p^h$ with $h\in \N$. Let $\mathcal B\subseteq\Z_q^n$ with $|\mathcal B|=q^n$ and $\mathcal B \mod p=\F_q^n$, and $\mathcal B=T_q^n(g_{ij}:j\in [1,n],i\in\N)$. Let $f\in \mathbb{Z}_q[x_1,\ldots,x_n]$ be a nonzero polynomial. Given an $m\in \N$, let
\begin{equation*}
  V:=\{X\in \mathcal B \mid  f(X)\equiv 0 \pmod {p^{m}}\}.
\end{equation*}
Write the Teichm\"uller expansion $f=\sum_{i=0}^{\infty} p^i f_i$ with $f_i=\sum_{j=1}^{r}a_{ij}X^{u_j}$. Let $d\in \N$. If for each term $a_{ij}X^{u_j}$, we have
\begin{equation}\label{degreesum}
\deg\left(a_{ij}\prod_{l=1}^{n}\prod_{t=1}^{\alpha_l}g_{_{{\beta_{tl}}l}}(X)\right)\leq dp^{h\lfloor\frac{i+|\beta|}{h}\rfloor}
\end{equation}
for all $i\in [0,m-1]$, $j\in [1,r]$,  $\beta=(\beta_{tl})\in \N_0^{|u_j|}$ with the sum $i+|\beta|\leq m-1$, then
\begin{equation*}
\mathrm{ord}_q(|V|)\geq \left\lceil\frac{n-\frac{p^{m}-1}{p-1}d}{p^{m-1}d}\right\rceil^*.
\end{equation*}
\end{theorem}
\begin{proof}
We follow the notations used in the proof of Theorem \ref{thm1}, so we do not explain them more. Choose an arbitrary term of $f$, say $a_{ij}X^up^i$ with $u=(\alpha_1,\dots,\alpha_n)\in \N_0^n$ and take any $\beta=(\beta_{tj})\in \N_0^{|u|}$. Let $Y=(y_1,\dots,y_n)$ with $y_j=\sum_{i=0}^{\infty}g_{ij}(X)p^i$ for $j\in[1,n]$. Substituting $Y$ for $X$ in this term, by \eqref{77cc} we have
\begin{equation*}
a_{ij}p^i\prod_{l=1}^{n}\prod_{t=1}^{\alpha_l}\left(g_{_{{\beta_{tl}}l}}(X)p^{\beta_{tl}}\right)=a_{ij}p^{i+|\beta|}\prod_{l=1}^{n}\prod_{t=1}^{\alpha_l}g_{_{{\beta_{tl}}l}}(X).
\end{equation*}
This is zero modulo $p^m$ if $i+|\beta|\geq m$. Thus, we can assume $i+|\beta|\leq m-1$.
Since $a^{p^h} = a^q = a$ for $a \in \mathbb{F}_q$, we deduce
$$\left(\widetilde{a}_{ij}\prod_{l=1}^{n}\prod_{t=1}^{\alpha_l}\widetilde{g}_{_{\beta_{tl}l}}(X)\right)^{p^{i+|\beta|}}=\left((\widetilde{a}_{ij}\prod_{l=1}^{n}\prod_{t=1}^{\alpha_l}\widetilde{g}_{_{\beta_{tl}l}}(X)\right)^{p^{i+|\beta|-h\lfloor\frac{i+|\beta|}{h}\rfloor}}.$$
By \eqref{degreesum}, we have
$$\deg\left(a_{ij}\prod_{l=1}^{n}\prod_{t=1}^{\alpha_l}g_{_{{\beta_{tl}}l}}(X)\right)^{p^{i+|\beta|-h\lfloor\frac{i+|\beta|}{h}\rfloor}}
\leq dp^{h\lfloor\frac{i+|\beta|}{h}\rfloor}{p^{i+|\beta|-h\lfloor\frac{i+|\beta|}{h}\rfloor}}= dp^{i+|\beta|}.$$
The remaining is similar to the proof of Theorem \ref{thm1}.
\end{proof}

A weak version is the following result.
\begin{corollary}[Weak Version]\label{thm33}
Let $p$ be a prime number and $q=p^h$ with $h\in \N$. Let $\mathcal B\subseteq\Z_q^n$ with $|\mathcal B|=q^n$ and $\mathcal B \mod p=\F_q^n$, and $\mathcal B=T_q^n(g_{ij}:j\in [1,n],i\in\N)$. Let $f\in \mathbb{Z}_q[x_1,\ldots,x_n]$ be a nonzero polynomial. Given an $m\in \N$, let
\begin{equation*}
  V:=\{X\in \mathcal B \mid  f(X)\equiv 0 \pmod {p^{m}}\}.
\end{equation*}
If $\deg(g_{ij})\leq
p^{h\lfloor\frac{i}{h}\rfloor}$ for $j\in [1,n],i\in [1, m-1]$, then
\begin{equation*}
\mathrm{ord}_q(|V|)\geq \left\lceil\frac{n-\frac{p^{m}-1}{p-1}\deg (f)}{p^{m-1}\deg (f)}\right\rceil^*.
\end{equation*}
\end{corollary}
\begin{proof}
  Suppose $\deg(g_{ij})\leq
p^{h\lfloor\frac{i}{h}\rfloor}$ for $j\in [1,n],i\in [1, m-1]$.
Let $d=\deg(f)$. By \eqref{77aa} we see that \eqref{degreesum} holds naturally for $\alpha=(\alpha_1,\dots,\alpha_n)\in\N_0^n$ and $\beta=(\beta_{tl})\in \N_0^{|\alpha|}$ with $|\beta| \leq m-1$.
\end{proof}

Corollary \ref{thm33} for the case $q=p$ becomes simpler because $h\lfloor\frac{i}{h}\rfloor=i$ for $h=1$.


\begin{corollary}[Weak Version]\label{cor6}
Let $p$ be a prime number. Let $\mathcal B\subseteq\Z_p^n$ with $|\mathcal B|=p^n$ and $\mathcal B \mod p=\F_p^n$, and $\mathcal B=T_p^n(g_{ij}:j\in [1,n],i\in\N)$. Let $f\in \mathbb{Z}_p[x_1,\ldots,x_n]$ be a nonzero polynomial. Given an $m\in \N$, let
\begin{equation*}
  V:=\{X\in \mathcal B \mid  f(X)\equiv 0 \pmod {p^{m}}\}.
\end{equation*}
If $\deg(g_{ij})\leq p^i$ for all $j\in [1,n],i\in [1, m-1]$, then
\begin{equation}\label{llwwhh}
\mathrm{ord}_p(|V|)\geq \left\lceil\frac{n-\frac{p^{m}-1}{p-1}\deg (f)}{p^{m-1}\deg (f)}\right\rceil^*.
\end{equation}
In particular, \eqref{llwwhh} holds true for all $\mathcal B$ in split form.
\end{corollary}

\subsection{For polynomial system}

We extend the results above to the system of polynomials without proofs.

\begin{theorem}[Strong Version]\label{thm4}
Let $p$ be a prime number and $q=p^h$ with $h\in \N$. Let $\mathcal B\subseteq\Z_q^n$ with $|\mathcal B|=q^n$ and $\mathcal B \mod p=\F_q^n$, and $\mathcal B=T_q^n(g_{ij}:j\in [1,n],i\in\N)$. Let $f_1,\dots, f_s\in \mathbb{Z}_q[x_1,\ldots,x_n]$ be a system of nonzero polynomials. For given $m_1,\dots,m_s\in \N$, let
\begin{equation*}
  V:=\{X\in \mathcal B \mid  f_k(X)\equiv 0 \pmod {p^{m_k}} \mbox{ for all } k\in [1,s]\}.
\end{equation*}
For each $k\in[1,s]$, write the $p$-adic Teichm\"uller expansion
$$f_k=\sum_{i=0}^{\infty} p^if_{k,i}(X), $$
with $f_{k,i}(X)=\sum_{j=1}^{r_k}a_{ij}^{(k)}X^{u_j^{(k)}}$. Let $d_1,\dots,d_s\in \N$. If for each term $a_{ij}^{(k)}X^{u_j^{(k)}}$, we have
\begin{equation*}
\deg\left(a_{ij}^{(k)}\prod_{l=1}^{n}\prod_{t=1}^{\alpha_l}g_{_{{\beta_{tl}}l}}(X)\right)\leq d_kp^{h\lfloor\frac{i+|\beta|}{h}\rfloor}
\end{equation*}
for all $i\in [0,m_k-1]$, $j\in [1,r_k]$,  $\beta=(\beta_{tj})\in \N_0^{|u_j^{(k)}|}$ with the sum $i+|\beta|\leq m_k-1$, then
\begin{equation*}
\mathrm{ord}_q(|V|)\geq \left\lceil\frac{n-\sum_{k=1}^s\frac{p^{m_k}-1}{p-1}d_k}{\max\nolimits_{k \in [1,s]}\{p^{m_k-1}d_k\}}\right\rceil^*.
\end{equation*}
\end{theorem}

A weaker consequence is the following

\begin{corollary}[Weak Version]\label{thm44}
Let $p$ be a prime number and $q=p^h$ with $h\in \N$. Let $\mathcal B\subseteq\Z_q^n$ with $|\mathcal B|=q^n$ and $\mathcal B \mod p=\F_q^n$, and $\mathcal B=T_q^n(g_{ij}:j\in [1,n],i\in\N)$. Let $f_1,\dots, f_s\in \mathbb{Z}_q[x_1,\ldots,x_n]$ be a system of nonzero polynomials. For given $m_1,\dots,m_s\in \N$, let
\begin{equation*}
  V:=\{X\in \mathcal B \mid  f_k(X)\equiv 0 \pmod {p^{m_k}} \mbox{ for all } k\in [1,s]\}.
\end{equation*}
If $\deg(g_{ij})\leq
p^{h\lfloor\frac{i}{h}\rfloor}$ for $j\in [1,n],i\in [1, m-1]$,
 then
\begin{equation*}
\mathrm{ord}_q(|V|)\geq \left\lceil\frac{n-\sum_{k=1}^s\frac{p^{m_k}-1}{p-1}\deg (f_k)}{\max\nolimits_{k \in [1,s]}\{p^{m_k-1}\deg (f_k)\}}\right\rceil^*.
\end{equation*}
\end{corollary}

In the case $q=p$,  the above corollary reduces to

\begin{corollary}[Weak Version]\label{cor8}
Let $p$ be a prime number. Let $\mathcal B\subseteq\Z_p^n$ with $|\mathcal B|=p^n$ and $\mathcal B \mod p=\F_p^n$, and $\mathcal B=T_p^n(g_{ij}: j\in [1,n],i\in\N)$. Let $f_1,\dots, f_s\in \mathbb{Z}_p[x_1,\ldots,x_n]$ be a system of nonzero polynomials. For given $m_1,\dots,m_s\in \N$, let
\begin{equation*}
  V:=\{X\in \mathcal B \mid  f_k(X)\equiv 0 \pmod {p^{m_k}} \mbox{ for all } k\in [1,s]\}.
\end{equation*}
If $\deg(g_{ij})\leq p^i$ for all $j\in [1,n],i\in [1, m-1]$, then
\begin{equation}\label{nov1}
\mathrm{ord}_p(|V|)\geq \left\lceil\frac{n-\sum_{k=1}^s\frac{p^{m_k}-1}{p-1}\deg (f_k)}{\max\nolimits_{k \in [1,s]}\{p^{m_k-1}\deg (f_k)\}}\right\rceil^*.
\end{equation}
In particular, \eqref{nov1} holds true for all boxes $\mathcal B$ in split form.
\end{corollary}

The corollary above shows that Theorem \ref{grythm} extends to any box $\mathcal B$ in split form when $q=p$. However, as illustrated in the examples below, Theorem \ref{grythm} cannot be extended to arbitrary $\mathcal{B}$ in general, even when $q=p$ (cf. Example \ref{exam1}) or $\mathcal B$ in split form if $q$ is a higher power of $p$, (cf. the last two rows in Table \ref{tabnb} in Example \ref{exam2}).

\begin{example}\label{exam1}
Let $p=2$. Let $f=x_1+x_2+x_3+x_4\in \Z_2[x_1,x_2,x_3,x_4]$.  Let $V=\{X\in T_2^4 \mid f(X)=0\}$. By Theorem \ref{thm1}, we have $\ord_p(|V|)\geq 1$ (in fact $\ord_p(|V|)=3$).  Given an $a\in T_2$, let $\mathcal B_a=\{(a_1+(a_1a_2a_3a_4+a)p,a_2,a_3,a_4)\mid a_i\in T_p, i\in [1,4]\}$ and $V_a=\{X\in \mathcal B_a \mid f(X)=0 \pmod {p^2}\}$.
Thus, we have $g_1=x_1x_2x_3x_4+a$, and $g_i=0$ for $i\in [2,4]$. Now we consider the cardinality of $V_a$, that is, the number of solutions of the congruence
\begin{equation*}
  (x_1+x_2+x_3+x_4)+(x_1x_2x_3x_4+a)p\equiv 0 \pmod {p^2}
\end{equation*}
with $x_i\in T_p$. By Lemma \ref{keylem} and Remark \ref{snmn}, it is equivalent to counting the number of solutions in $\F_2$ of the system
 \begin{equation*}
\left\{
  \begin{array}{ll}
    y_1+y_2+y_3+y_4=0, \\
    y_1y_2y_3y_4+\widetilde{a}-\sum\frac{1}{2}{2 \choose {t_1,\dots,t_{4}}}y_1^{t_1}\cdots y_{4}^{t_{4}}=0,
  \end{array}
\right.
\end{equation*}
where $y_i=\widetilde{x}_i$ for $i\in [1,4]$ and the sum in the second equation is over all the tuples $(t_1,\dots,t_{4})$ satisfying that $t_1+\dots+t_{4}=2$ and $0\leq t_i<2$ for all $i$. By easy calculation, we get $|V_0|=1$ and $|V_1|=7$ respectively, to both of which Theorem \ref{thm1} cannot be applied as $\deg(g_1)=4$ is larger than $p=2$.
\end{example}

\begin{example}\label{exam2}
Let $p=3$ and $q=p^2=9$. Let $f=x_1+\cdots+x_5\in \Z_q[x_1,\ldots,x_5]$. Given a vector $u=(d_1,\dots,d_5)\in \N_0^5$, let the box $\mathcal B_{(d_1,\dots,d_5)}$ be the set $\{(a_1+a_1^{d_1}p,\dots,a_5+a_5^{d_5}p)\mid a_i\in T_q, i\in [1,5]\}$, or in concise notation, $\mathcal B_u=\{X+X^up\mid X\in T_q\}$.
 Define $V_u:=\{X\in \mathcal B_u \mid f(X)\equiv 0 \pmod {p^2}\}$. By Theorem \ref{thm1}, we have $\ord_p(|V_\mathbf{0}|)\geq 2$ (in fact $\ord_p(|V_\mathbf{0}|)=8$). Now we consider the cardinality of $V_u$ for $u=(d_1,\dots,d_5)\neq \mathbf{0}$, that is, the number of solutions of the congruence
\begin{equation*}
  (x_1+\cdots+x_5)+(x_1^{d_1}+\cdots+x_5^{d_5})p\equiv 0 \pmod {p^2}
\end{equation*}
with $x_i\in T_q$. By Lemma \ref{keylem} and Remark \ref{snmn}, it is equivalent to counting the number of solutions in $\F_q$ of the system
 \begin{equation*}
\left\{
  \begin{array}{ll}
    y_1+\cdots+y_{5}=0, \\
    y_1^{3d_1}+\dots+y_{5}^{3d_{5}}-\sum\frac{1}{3}{3 \choose {t_1,\dots,t_{5}}}y_1^{t_1}\cdots y_{5}^{t_{5}}=0,
  \end{array}
\right.
\end{equation*}
where $y_i=\widetilde{x}_i$ for $i\in [1,5]$ and the sum in the second equation is over all the tuples $(t_1,\dots,t_{5})$ satisfying that $t_1+\dots+t_{5}=3$ and $0\leq t_i<3$ for all $i$. Randomly choosing some vectors $u$ in $\N_0^5$ in which some components are greater than $p$ (so Theorem \ref{thm1} is not valid for them), and computing via computer, we get the results listed in Table \ref{tabnb}.

\begin{table}[h]
\begin{center}
\caption{}
\label{tabnb}
\begin{tabular}{|c|c|c|c|}
  \hline
  $u$ & $3u \mod q-1$  &$|V_u|$ & $\ord_p(|V_u|)$ \\
  \hline
  $(4,4,4,4,4)$ &$(4,4,4,4,4)$  & $1206$ & $2$ \\
  \hline
  $(5,5,5,5,5)$ &$(7,7,7,7,7)$  & $2601$ & $2$ \\
  \hline
  $(6,6,6,6,6)$ &$(2,2,2,2,2)$  & $864$ & $3$ \\
  \hline
  $(7,7,7,7,7)$ &$(5,5,5,5,5)$  & $1881$ & $2$ \\
  \hline
  $(8,8,8,8,8)$ &$(8,8,8,8,8)$  & $606$  & $1$ \\
  \hline
  $(4,7,2,5,8)$ &$(4,5,6,7,8)$  & $660$  & $1$ \\
  \hline
\end{tabular}
\end{center}
\end{table}
\end{example}

The last two rows in Table \ref{tabnb} show that Theorem \ref{thm4} is false without the degree bound condition on $g_{ij}$, even for split boxes.

\section*{Acknowledgements}
The authors thank Weihua Li for providing the data in Example \ref{exam2} by computer program. The first author is jointly supported by the National Natural Science Foundation of China (Grant No. 11871291), Natural Science Foundation of Fujian Province, China (No. 2022J02046).

\end{document}